\documentclass[12pt]{amsart}
\usepackage{amsmath,amsfonts,amssymb,amsthm}
\usepackage{pinlabel}
\usepackage{graphicx}
\usepackage{mathtools}
\usepackage[all]{xy}
\usepackage{hyperref}
\usepackage[usenames,dvipsnames]{color}
\usepackage{color}
\xyoption{dvips}

\newtheorem{theorem}{Theorem}[section]

\newtheorem{proposition}[theorem]{Proposition}

\theoremstyle{remark}
\newtheorem{remark}[theorem]{Remark}
\theoremstyle{definition}

\newcommand{\C}{\mathbb{C}}
\newcommand{\R}{\mathbb{R}}
\newcommand{\Z}{\mathbb{Z}}
\newcommand{\Q}{\mathbb{Q}}
\newcommand{\A}{\mathcal A}
\newcommand{\mQ}{\mathcal{Q}}

\newcommand{\F}{\mathbb{F}}

\numberwithin{equation}{section}

\begin{document}

\title[On non-geometric augmentations in high dimensions]{On non-geometric augmentations in high dimensions}

\author{Roman Golovko}

\begin{abstract}
In this note we construct augmentations of Chekanov-Eliashberg algebras of certain high dimensional Legendrian submanifolds that are not induced by exact Lagrangian fillings.
The obstructions to the existence of exact Lagrangian fillings that we use are Seidel's isomorphism and the injectivity of a certain algebraic map between the corresponding augmentation varieties proven by Gao and Rutherford. In addition, along the way we discuss the relation between augmentation varieties of Legendrian submanifolds and their spherical spuns (Proposition \ref{augeqaug0}).
\end{abstract}

\address{Faculty of Mathematics and Physics, Charles University, Sokolovsk\'{a} 83, 18000 Praha 8, Czech Republic}
\email{golovko@karlin.mff.cuni.cz}
\date{\today}
\thanks{}


\maketitle

\section{Introduction and main results}

It is natural to study symplectic manifolds with contact boundary by studying Lagrangian submanifolds with Legendrian boundary; in particular one can study exact Lagrangian fillings of Legendrian submanifolds.
In this paper we consider closed Legendrian submanifolds $\Lambda$ in the standard contact vector space $\R^{2n+1}_{st}:=(\R^{2n+1}, \alpha_{st}:=dz-\Sigma_i y_idx_i)$ and their exact Lagrangian fillings, i.e. smooth cobordisms $(L; \emptyset,
\Lambda)$ and Lagrangian embeddings $L\hookrightarrow(\R\times \R^{2n+1}, d(e^t\alpha_{st}))$ satisfying
$
L|_{[T_{L},\infty)\times \R^{2n+1}} = [T_{L},\infty)\times \Lambda
$
for some $T_{L}\gg 0$, $L^{c}:=L|_{(-\infty, T_{L}]\times
\R^{2n+1}}$ is compact, and there is $f\in C^{\infty}(L)$ which is constant on $[T_{L},\infty) \times \Lambda$ satisfying $df=e^t\alpha_{st}$.

Legendrian contact homology is a modern invariant of
Legendrian submanifolds in  $\R^{2n+1}_{st}$ which is a variant of the
symplectic field theory (SFT) introduced by Eliashberg, Givental, and
Hofer in \cite{EGHSFT}. For Legendrian submanifolds in
$\R^3_{st}$, it was defined by Chekanov in \cite{CHDGA} and the version of Legendrian contact homology for
Legendrian submanifolds of $\R^{2n+1}_{st}$ was developed by
Ekholm--Etnyre--Sullivan in \cite{EESLSIRHD, EkholmEtnyreSullivannoniso}.

Legendrian contact homology is a homology of the differential graded algebra (DGA) $\A(\Lambda)$, which is called the {\em Chekanov-Eliashberg DGA of $\Lambda$} or {\em Legendrian contact homology DGA of $\Lambda$}.
Chekanov--Eliashberg DGA $\A(\Lambda)$ is defined to be the non-commutative unital differential graded algebra freely generated by the set of Reeb chords of $\Lambda$, denoted by $\mQ(\Lambda)$, it is defined over a unital ring $R$.
The differential $\partial(a)$ on $a \in \mQ(\Lambda)$ is defined  by a count of rigid pseudo-holomorphic disks for some choice of compatible almost complex structure, and is then extended using the Leibniz rule. The homology of the Chekanov-Eliashberg algebra is
called the {\em
Legendrian contact homology of $\Lambda$}. Following the result of Ekholm--Etnyre--Sullivan in \cite{EESLSIRHD}, Legendrian contact homology is independent of the
choice of an almost complex structure and is invariant under Legendrian
isotopy. When $\A(\Lambda)$ is defined over a unital ring $R$, we will sometimes write it as $\A_{R}(\Lambda)$.

An {\em augmentation} is a unital DGA-morphism
$\varepsilon : \mathcal  A(\Lambda) \to (\F, 0)$
which thus satisfies $\varepsilon\circ \partial = 0$. In our constructions we will use only graded augmentations  which
by definition vanish on all generators in nonzero degrees.

The following fact can be seen as a consequence of the discussion in \cite{EkhomHonaKalmancobordisms}: An exact Lagrangian filling $L$ of $\Lambda$
gives rise to a unital DGA morphism
$\varepsilon: \mathcal A(\Lambda) \to \F_2$
defined by an appropriate count of rigid pseudoholomorphic discs with boundary on the
filling. Even though the result in \cite{EkhomHonaKalmancobordisms} is written for $\F_2$-coefficients only, it admits a
natural extension to more general coefficients. In this work we will only consider spin Maslov number $0$ exact Lagrangian cobordisms of  spin Maslov number $0$ Legendrian submanifolds of the standard contact vector space, where we will always choose the spin structure on a Legendrian which is a restriction of the spin structure of the exact Lagrangian filling. In this case there is a  map $\varepsilon: \mathcal A(\Lambda) \to \F$ for an arbitrary field $\F$, and all the homology groups that we will consider will have a $\Z$-grading.

There are a few obstructions to the existence of an exact Lagrangian filling which induces a given augmentation, see \cite{ChantraineLagrConc, DimitroglouRizellLiftingPSH, EkholmRatSFT, EkholmLekiliduality, nonfilLegkninthe3sph, NFAugoTK, GolovkonoteLagrCob}, and quite a few examples of augmentations of Legendrian knots that are not induced by exact Lagrangian fillings (i.e.,  are {\em non-geometric}).
We would like to construct non-geometric augmentations  for high dimensional Legendrian submanifolds.

\subsection{Seidel's isomorphism}

One of the most well-known obstruction comes from the so-called Seidel's isomorphism.
In \cite{EkholmRatSFT}, Ekholm outlined an isomorphism, first conjectured by Seidel,
which relates the linearised Legendrian contact cohomology of a Legendrian and the singular homology of its embedded exact Lagrangian filling. The details of this isomorphism were later completed in the work of Dimitroglou Rizell, see
\cite{DimitroglouRizellLiftingPSH}:

\begin{theorem}[\cite{DimitroglouRizellLiftingPSH, EkholmRatSFT}]
\label{SeidelIso}
Given a Legendrian submanifold $\Lambda\subset \R^{2n+1}$ and its exact Lagrangian filling $L$ of Maslov number $0$. For the augmentation $\varepsilon_L:\A_{\F_2}(\Lambda)\to (\F_2, 0)$ induced by $L$, there is an isomorphism
$$LCH^{n-i}_{\varepsilon_L}(\Lambda;\F_2)\simeq H_i(L; \F_2).$$
\end{theorem}

\begin{remark}
\label{extSeiiso}
Note that Seidel's isomorphism has been proven by  Dimitroglou Rizell in \cite{DimitroglouRizellLiftingPSH} only with $\F_2$-coefficients. On the other hand, signs of the cobordism maps have been worked out by  Karlsson in \cite{KarlssonOrcob}. Besides that the stronger result of Ekholm--Lekili \cite{EkholmLekiliduality} which compares the corresponding $A_\infty$-structures has been proven with signs, in particular it works over an arbitrary field. From this perspective we can say that Seidel's isomorphism holds with $\Q$-coefficients and with $\Z$-coefficients.
In addition, observe that there is a version of Seidel's isomorphism with local coefficients recently proven by Gao--Rutherford, see \cite{NFAugoTK}.
\end{remark}

\subsection{Obstruction of fillability from augmentation varieties}
\label{Ofaugvarsec}
In \cite{NFAugoTK} Gao and Rutherford use local structure of the augmentation variety to obstruct Lagrangian
fillings, and to provide the examples of Legendrian twist knots with augmentations for which the other known obstructions to the existence of an exact Lagrangian filling inducing it such as the one coming from  Thurston-Bennequin number \cite{ChantraineLagrConc}, Seidel's isomorphism \cite{DimitroglouRizellLiftingPSH, EkholmRatSFT}, the extension of Seidel's isomorphism of Ekholm-Lekili \cite{EkholmLekiliduality} and the examples of Etg\"{u} \cite{nonfilLegkninthe3sph} based on the result of Ekholm-Lekili do not apply.

Now we recall some details of the obstruction of Gao and Rutherford from \cite{NFAugoTK}.
Let $L$ be an exact Lagrangian filling of a Legendrian submanifold $\Lambda\subset \R^{2n+1}_{st}$, then
$$H_1(L;\Z)\simeq \Z^k\oplus \Z/k_1 \oplus \dots \oplus \Z/k_s$$
for some $k,k_1,\dots,k_s$. We define $$Aug(L;\mathbb F)\simeq (\mathbb F^{\ast})^k\oplus C_{k_1}(\mathbb F) \oplus \dots \oplus C_{k_r}(\mathbb F),$$
where $C_{k_i}(\mathbb F)$ is the group of $k_i$-th roots of unity in $\mathbb F$, i.e.  $$C_{k_i}(\mathbb F)=\{ x\ |\ x^{k_i}=1, x\in \mathbb F\}$$ and $i=1,\dots,k_s$.
\begin{proposition}[Proposition 2.6 in \cite{NFAugoTK}]
\label{nmamforb}
Let  $L$ be an exact Lagrangian filling of a Legendrian submanifold $\Lambda\subset (\R^{2n+1}, \alpha_{st})$  such that the Maslov number of $L$ vanishes. If $\mathbb F$ has a characteristic different from $2$, assume that $L$ is equipped with a choice of spin structure.
Then, the map
$f^{\ast}_L : Aug (L, \mathbb F) \to Aug (\Lambda, \mathbb F)$
is an injective, algebraic map.
\end{proposition}
\begin{remark}
Note that when $\Lambda$ is a Legendrian knot in $(\R^3, \alpha_{st})$ and $L$ is an orientable exact Lagrangian filling of $\Lambda$, then $H_1(L;\Z)$ is a free abelian group. Hence $Aug (L, \mathbb F)\simeq (\mathbb F^{\ast})^k$ for some $k$.
\end{remark}

\subsection{Main result}

Our main result is the extension of the examples in $\R^3_{st}$ of augmentations whose geometricity is obstructed by Seidel's isomorphism (see Section \ref{classAlowdim}) to high dimensions (see Section \ref{HDAOCA}) and the extension of the examples in $\R^3_{st}$ of  augmentations whose geometricity is not obstructed by Seidel's isomorphism, but  is obstructed by Proposition \ref{nmamforb}  (see Section \ref{secclassBGR}) to high dimensions (see Section \ref{NonSeiClBHD}). 

The examples we get will be obtained using the spherical spinning construction.

\subsubsection*{Spherical spinning}

Recall that the front  $S^k$-spinning construction is a Legendrian version of suspension.  The front  $S^k$-spinning construction from a Legendrian submanifold $\Lambda\subset \R^{2n+1}_{st}$ produces a Legendrian embedding of $\Lambda \times S^k$ inside $\R^{2(n+k)+1}_{st}$ whose image is denoted by $\Sigma_{S^k} \Lambda$. When $k=1$ it has appeared in the work of Ekholm--Etnyre--Sullivan \cite{EkholmEtnyreSullivannoniso} and in the case when $k>1$ it has been discussed by the author in \cite{GolovkoSphericalSpinning}. One of the important properties of the spherical front spinning that we will need is that as shown in \cite{GolovkoSphericalSpinning} it can be extended to exact Lagrangian cobordisms.  For other details of the construction and its properties we refer the reader to
\cite{DRGEstimating, EkholmEtnyreSullivannoniso,  GolovkoSphericalSpinning, GolovkoInf22, GolovkoTopDistInf22}.

Both types of examples (from Sections  \ref{HDAOCA},   \ref{NonSeiClBHD}) are joined in the following theorem.
\begin{theorem}
\label{topresnongaug}
There is a Legendrian submanifold $\Lambda$ in $\mathbb R^{2n+1}_{st}$ of Maslov number $0$ such that the Chekanov-Eliashberg algebra of $\Lambda$ admits an augmentation
$\varepsilon: \A(\Lambda)\to (\F_2, 0)$ which is not induced by a spin exact Lagrangian filling of Maslov number $0$.
\end{theorem}

In addition, along the way we discuss the relation between augmentation varieties of Legendrian submanifolds and their spherical spuns (Proposition \ref{augeqaug0}).

\section{Examples in low dimensions}
\begin{figure}[!ht]
\includegraphics[scale=.9]{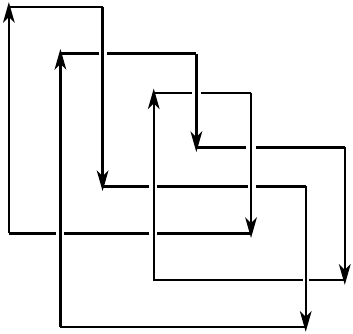}
\vspace{0.1in}
\caption{The grid diagram of the Legendrian representative of $m(8_{21})$ from \cite{ChNgAtlas}.}
\label{fig:Seim821}
\end{figure}
\subsection{Class A}
\label{classAlowdim}
In this class we consider Legendrian knots $\Lambda$ in $\R^3_{st}$ such that the Chekanov--Eliashberg algebra 
admits an augmentation $$\varepsilon: \A(\Lambda)\to (\F_2,0)$$ satisfying that for some $i>1$ or $i<0$, $LCH_{\varepsilon}^i(\Lambda;\F_2)$ is not isomorphic to $H_{1-i}(L_{\Lambda};\F_2)$ for all exact Lagrangian fillings  $L_{\Lambda}$ of Maslov number $0$ .
In other words, in this class we consider Legendrian knots with augmentations that are not geometric because they violate the obstruction coming from Seidel's isomorphism.
\begin{remark}
From Remark \ref{extSeiiso} it follows that  in the description of Class A we can rely on Seidel's isomorphism not only with $\F_2$-coefficients, but with more general field coefficients such as  $\mathbb Q$ and $\mathbb R$, and also with $\Z$-coefficients.
\end{remark}
We rely on the work of Chongchitmate--Ng  \cite{ChNgAtlas}:
\begin{itemize}
\item[(i)] Legendrian representative of $m(8_{21})$ from \cite{ChNgAtlas}, see Figure \ref{fig:Seim821}, has a vanishing rotation number, and hence Maslov number $0$, and two Poincar\'{e} polynomials, one of which is of the form $$P_{m(8_{21})}(t)=t^{-1} + 4 + 2t,$$
The augmentation $\varepsilon_{m(8_{21})}$ which corresponds to this polynomial has the property that
$$LCH^{\varepsilon_{m(8_{21})}}_{-1}(\Lambda;\F_2)\simeq LCH_{\varepsilon_{m(8_{21})}}^{-1}(\Lambda;\F_2)\simeq \F_2.$$
From Theorem \ref{SeidelIso} it follows that $\varepsilon_{m(8_{21})}$ is not induced by a Maslov number $0$ exact Lagrangian filling $L$, since otherwise $H_{2}(L;\mathbb F_2)\simeq \F_2$, which is impossible from the topological point of view.
\item[(ii)] We can argue the same way for many other Legendrian knots $\Lambda$  (in particular from the atlas in \cite{ChNgAtlas}), whose Poincar\'{e} polynomials are computed with respect to certain augmentations that are by topological reasons in conflict with Seidel's isomorphism. For example, one can take Legendrian representatives of $9_{46}$, $9_{48}$ and so on from \cite{ChNgAtlas}.
\end{itemize}

\subsection{Class B}
\label{secclassBGR}
Consider the following Legendrian representatives of twist knots $\Lambda_n$, $n$ is odd and $n=2k+1>3$, investigated  by Rutherford and Gao in \cite{NFAugoTK}.
The set of Reeb chords $\mQ(\Lambda_n)$ consists of chords $a,b$, $c_1,\dots, c_n$ and $e_0, e_1, \dots e_n$, see Figure \ref{fig:family}.
Reeb chords $a$, $b$ and $c_1, \ldots, c_k$ appear on the right, and crossings $c_{k+1}, \ldots, c_{n}$ appear on the left.  The right cusps are labeled in clockwise order starting at the upper right as $e_0, e_1, \ldots, e_{k+1}$ (appearing on the right) and $e_{k+2}, \ldots, e_n$ (appearing on the left).

\begin{figure}[!ht]
\labellist
\small
\pinlabel $c_1$ [l] at 195 153
\pinlabel $c_2$ [l] at 195 115
\pinlabel $c_3$ [l] at 195 75
\pinlabel $c_4$ [l] at 40 85
\pinlabel $c_5$ [l] at 40 123
\pinlabel $c_6$ [l] at 40 161
\pinlabel $c_7$ [l] at 48 197
\pinlabel $a$ [b] at 156 193
\pinlabel $b$ [b] at 194 201
\pinlabel $e_0$ [l] at 220 223
\pinlabel $e_1$ [l] at 220 173
\pinlabel $e_2$ [l] at 220 137
\pinlabel $e_3$ [l] at 220 99
\pinlabel $e_4$ [l] at 220 53
\pinlabel $e_5$ [l] at 68 103
\pinlabel $e_6$ [l] at 68 143
\pinlabel $e_7$ [l] at 68 178
\endlabellist
\includegraphics[scale=.7]{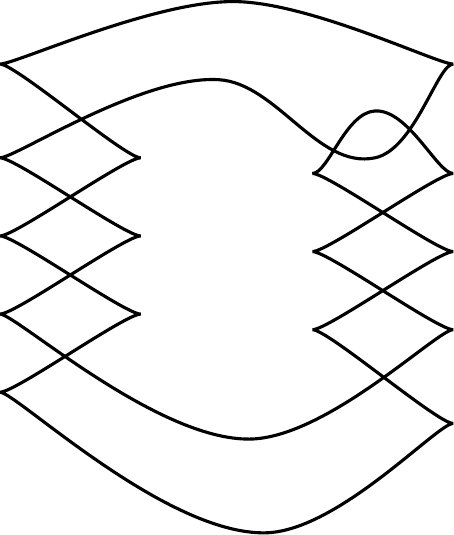}
\vspace{0.1in}
\caption{The front projection of $\Lambda_7$.}
\label{fig:family}
\end{figure}

As proven in \cite{NFAugoTK}, there is an augmentation $\varepsilon_{\Lambda_n}$ to $\F_2$ defined by
$$\varepsilon_{\Lambda_n}(a)=0, \varepsilon_{\Lambda_n}(b)=0, \varepsilon_{\Lambda_n}(c_i)=1, \varepsilon_{\Lambda_n}(e_j)=0$$
for all $i=1,\dots,n$ and $j=0,\dots,n$, that is not induced by any Maslov $0$ exact Lagrangian filling.

\section{High dimensional analogue of Class A}
\label{HDAOCA}
We now consider the collection of inductive spuns of the Legendrian representative $\Lambda$ of $m(8_{21})$ that we discussed in Section \ref{classAlowdim}.
Following the discussion in \cite{DRGEstimating} observe that there is an inclusion of  DGAs $i:\mathcal A(\Lambda) \hookrightarrow \mathcal A(\Sigma_{S^1}\Lambda)$, which can be left inverted by a surjective DGA map $\pi:\mathcal A(\Sigma_{S^1}\Lambda)\to \mathcal A(\Lambda)$.

\begin{theorem}
\label{proponespinfoeoneex}
Given $\Sigma_{S^1}\dots \Sigma_{S^1} \Lambda$ and an augmentation $\varepsilon_{\Sigma_{S^1}\dots \Sigma_{S^1} \Lambda}$ of $\A_{\F_2}(\Sigma_{S^1}\dots \Sigma_{S^1} \Lambda)$ which is given by the inductive application of $\pi^{\ast}$ to $\varepsilon_{\Lambda}$. Then $\varepsilon_{\Sigma_{S^1}\dots \Sigma_{S^1}  \Lambda}$ is not induced by a Maslov number $0$  exact Lagrangian filling.
\end{theorem}
\begin{proof}
First we consider the case of $S^1$-spun of $\Lambda$. From the K\"{u}nneth-type formula described in \cite[Theorem 4.1, Remark 4.2]{GolovkoRizellsphsp}, it follows that $$LCH^{\varepsilon_{\Sigma_{S^1} \Lambda}}_{i} (\Sigma_{S^1}\Lambda; \F_2) \simeq LCH^{\varepsilon_{\Lambda}}_{i} (\Lambda; \F_2)\oplus  LCH^{\varepsilon_{\Lambda}}_{i-1}(\Lambda; \F_2).$$
Then using it and the fact that $LCH^{\varepsilon_{\Lambda}}_{-1}(\Lambda; \F_2)\simeq \F_2$, $LCH^{\varepsilon_{\Lambda}}_{-2}(\Lambda; \F_2)\simeq 0$, we obtain
\begin{align}
\label{ineq:1}
\dim (LCH^{\varepsilon_{\Sigma_{S^1} \Lambda}}_{-1} (\Sigma_{S^1}\Lambda; \F_2))= 1.
\end{align}
Assume that $\varepsilon_{\Sigma_{S^1} \Lambda}$ is induced by the Maslov number $0$ spin exact Lagrangian cobordism $L$. Then from Theorem \ref{SeidelIso} we see that $$LCH_{2-i}^{\varepsilon_{\Sigma_{S^1} \Lambda}}(\Sigma_{S^1} \Lambda;\F_2)\simeq LCH^{2-i}_{\varepsilon_{\Sigma_{S^1} \Lambda}}(\Sigma_{S^1} \Lambda;\F_2)\simeq H_i(L; \F_2),$$
and hence from Equation \ref{ineq:1} it follows that $\dim H_3(L; \F_2)= 1$, which is impossible since $L$ is a spin $3$-dimensional filling of $\Sigma_{S^1}\Lambda$.

Then we inductively apply the same strategy and get that $\Sigma_{S^1}\dots \Sigma_{S^1} \Lambda$ admits an augmentation, $\varepsilon_{\Sigma_{S^1}\dots \Sigma_{S^1} \Lambda}$, which is given by the inductive application of  $\pi^{\ast}$ on $\varepsilon_{\Lambda}$, and which is not induced by any spin, Maslov number $0$ exact Lagrangian cobordism.
\end{proof}

\begin{remark}
Other Legendrian knots from Class A can be treated in the absolutely analogous way, i.e. the result of Theorem \ref{proponespinfoeoneex} will work for other Legendrian knots from Class A.
\end{remark}

\begin{remark}
Combining the geography/realization result for Poincar\'{e} polynomials by Bourgeois--Galant \cite{BourgeoisGalant19} with Seidel's isomorphism, one can find more non-geometric augmentations in the high dimensional analogue of Class A.   
\end{remark}


\section{High dimensional analogue of Class B}
\label{NonSeiClBHD}

In this section we consider  the  collection of inductive spuns of the Legendrian twist knots $\Lambda_n$ from Section \ref{secclassBGR}. Again, following the discussion in \cite{DRGEstimating} observe that there is an inclusion of  DGAs $i:\mathcal A(\Lambda_n) \hookrightarrow \mathcal A(\Sigma_{S^l}\Lambda_n)$, which can be left inverted by a surjective DGA map $\pi:\mathcal A(\Sigma_{S^l}\Lambda_n)\to \mathcal A(\Lambda_n)$ for $l\geq 1$.
\begin{theorem}
\label{proponespinfortwist}
For a Legendrian representative $\Lambda_n$, $\Sigma_{S^{1}}\dots \Sigma_{S^{1}} \Lambda_n$ admits a graded augmentation $\varepsilon_{\Sigma_{S^{1}}\dots \Sigma_{S^{1}} \Lambda_n}$ of $\A_{\F_2}(\Sigma_{S^{1}}\dots \Sigma_{S^{1}} \Lambda_n)$ which  is defined by the inductive application of $\pi^{\ast}$ to $\varepsilon_{\Lambda_n}$ and which is not induced by a Maslov number $0$ spin exact Lagrangian cobordism.
\end{theorem}

\begin{proof}

First we consider the case of $S^1$-spun of $\Lambda_n$. We take the graded augmentation $\varepsilon_{\Sigma_{S^1} \Lambda_n}$ of $\mathcal A(\Sigma_{S^{1}}\Lambda_n)$ defined by $\varepsilon_{\Sigma_{S^1} \Lambda_n}:=\pi^{\ast}(\varepsilon_{\Lambda_n})$, for simplicity we will denote $\varepsilon_{\Lambda_n}$ by $\varepsilon$ and $\varepsilon_{\Sigma_{S^1} \Lambda_n}$ by $\tilde{\varepsilon}$.
Now we prove that  $\tilde{\varepsilon}$ is not induced by a Maslov $0$, spin
exact Lagrangian cobordism $\tilde{L}$ of $\Sigma_{S^{1}} \Lambda_n$.

From the proof of \cite[Proposition 4.1]{NFAugoTK} we know that
\begin{align}
\label{eq2wndd}
LCH^{0}_{\varepsilon}(\Lambda_n; \mathbb F)\simeq \F^2\ \mbox{and} \ LCH^{1}_{\varepsilon}(\Lambda_n; \mathbb F)\simeq \F
\end{align}
for a field $\F$. Now we claim that $LCH^{0}_{\varepsilon}(\Lambda_n; \mathbb Z)\simeq \Z^2$, $LCH^{1}_{\varepsilon}(\Lambda_n; \mathbb Z)\simeq \Z$.
One way to see that
is to perform a direct computation of $LCH^{0}_{\varepsilon}(\Lambda_n; \mathbb Z)$, which is similar to the discussion in \cite{NFAugoTK}. We would like to thank Dan Rutherford for discussing the following computation.

We consider the augmentation $\varepsilon$, with  $\varepsilon(a) = \varepsilon(b) =0$.
We will base our consideration on \cite[Formulas (4.1) -- (4.4)]{NFAugoTK}.

With $\varepsilon(a)=\varepsilon(b)=0$, the value of $\varepsilon$ on the other degree $0$ generators must be
\begin{align*}
\varepsilon(c_i) = \left \{
\begin{array}{ll}
1, & \mbox{for} \  i \ \mbox{is odd};\\
-1, & \mbox{for}\ i \ \mbox{is even}.\\
\end{array}
\right.
\end{align*}

Then we specialize $t = \varepsilon(t) =-1$, and then take the word length $1$ part of the conjugated differential
$\Phi_{\varepsilon}\circ \partial\circ \Phi_{\varepsilon}^{-1}$, where $\Phi_{\varepsilon}(c) =  c+\varepsilon(c)$
for any Reeb chord $c$.

After that we see that the linearized differential is defined by $\partial_{\varepsilon}(e_0)=c_n$,  $\partial_{\varepsilon}(e_1)=-c_1$ and for $i>1$
\begin{align*}
\partial_{\varepsilon}(e_i) = \left \{
\begin{array}{ll}
c_{i-1}-c_{i}, & \mbox{for} \  i \ \mbox{is odd};\\
-c_{i-1}+c_{i}, & \mbox{for}\ i \ \mbox{is even}.\\
\end{array}
\right.
\end{align*}
Then we see that there is an isomorphism of free $\Z$-modules $$\partial_{\varepsilon}|_{\Z\langle e_1, e_2, \ldots, e_n\rangle}:\Z\langle e_1, e_2, \ldots, e_n\rangle\to \Z\langle c_1, c_2, \ldots, c_n\rangle,$$ and hence we can see that there is a acyclic subcomplex $\mathcal C$
$$0 \to \Z\langle e_1, e_2, \ldots, e_n\rangle \to \Z\langle c_1, c_2, \ldots, c_n\rangle \to 0,$$
and we can take a quotient of $LCC^{\varepsilon}/\mathcal C$, where the quotient map $LCC^{\varepsilon} \to LCC^{\varepsilon}/\mathcal C$ is a quasi-isomorphism, which leads to the complex $0 \to \Z\langle e_0\rangle \to \Z\langle a,b\rangle \to 0$ with the vanishing differential which leads to $LCH^{\varepsilon}_0=\Z^2$ and
$LCH^{\varepsilon}_1=\Z$, which using the universal coefficient theorem implies that  $LCH_{\varepsilon}^0=\Z^2$ and
$LCH_{\varepsilon}^1=\Z$

The alternative way to get the same result is to observe that
from Isomorphisms \ref{eq2wndd} it follows that

$$LCH^{0}_{\varepsilon}(\Lambda_n; \mathbb C)\simeq \C^2\ \mbox{and}\ LCH^{1}_{\varepsilon}(\Lambda_n; \mathbb C)\simeq \C,$$ and therefore the rank of
$ LCH^{0}_{\varepsilon}(\Lambda_n; \mathbb Z)$ equals $2$ and the rank of
$ LCH^{1}_{\varepsilon}(\Lambda_n; \mathbb Z)$ equals $1$. In order to avoid $p$-torsion, one can use computations  with a field of characteristic $p$.

From the K\"{u}nneth-type formula described in \cite[Theorem 4.1, Remark 4.2]{GolovkoRizellsphsp}, it follows that $$LCH_{\tilde{\varepsilon}}^{i} (\Sigma_{S^1}\Lambda; \Z) \simeq LCH_{\varepsilon}^{i} (\Lambda_n; \Z)\oplus  LCH_{\varepsilon}^{i-1}(\Lambda_n; \Z).$$ 

Hence $LCH_{\tilde{\varepsilon}}^{1}(\Sigma_{S^1}\Lambda_n)\simeq \Z^3$.
Then we use the Seidel's isomorphism over $\Z$ and see that $H_1(\tilde{L})\simeq \Z^3$, and hence from the discussion in Section \ref{Ofaugvarsec} we get that
\begin{align}
\label{augforspun}
Aug(\tilde{L},\F)\simeq(\F^{\ast})^3.
\end{align}

\begin{remark}
\label{stabh1}
Note that if we apply the argument described above to the $m$-th iterated spun of $\Lambda_n$, i.e. $\Sigma_{S^{1}}\dots \Sigma_{S^{1}} \Lambda_n$, then 
$H_1(\tilde{L})\simeq \Z^{2+m}$ and $Aug(\tilde{L},\F)\simeq(\F^{\ast})^{2+m}$.
\end{remark}

Then we need the following proposition:
\begin{proposition}
\label{augeqaug0}
Let $\Lambda$ be a Maslov number $0$ spin Legendrian submanifold such that $\mathcal Q_i(\Lambda)=\emptyset$ for all $i< 0$.
Then there is an isomorphism of (graded) augmentation varieties $Aug(\Lambda;\mathbb F)\simeq Aug(\Sigma_{S^m}\Lambda; \mathbb F)$ for all $m\geq 2$,
and  $Aug(\Sigma_{S^1}\Lambda; \mathbb F)\simeq Aug(\Lambda;\mathbb F)\times \mathbb F^{\ast} $
\end{proposition}
\begin{proof}
We will again rely on the analysis from \cite{DRGEstimating}.
Since $\mathcal A(\Lambda)$ is supported in non-negative degrees, we see that $\mathcal A(\Sigma_{S^m}\Lambda)$ is supported in non-negative degrees.
There is an inclusion of DGAs $i: \mathcal A(\Lambda)\to \mathcal A(\Sigma_{S^m}\Lambda)$ which admits a left-inverse  $\pi: \mathcal A(\Lambda)\to \mathcal A(\Sigma_{S^m}\Lambda)$ that has been proven with $\F_2$-coefficients, but admits a natural extension to group ring coefficients. 
We now assume that $m\geq 2$. Then $\Z[\pi_1(\Lambda)]\simeq \Z[\pi_1(S^m\times \Lambda)]\simeq \Z[\pi_1(\Sigma_{S^m}\Lambda)]$, and hence the coefficients of the corresponding DGAs are equivalent.
Since both $\mathcal A(\Lambda)$ and $\mathcal A(\Sigma_{S^m}\Lambda)$ are supported in non-negative degrees, the maps $i^{\ast}$, $\pi^{\ast}$ between graded augmentations of  $\mathcal A(\Lambda)$ and of  $\mathcal A(\Sigma_{S^m}\Lambda)$ induced by $i$ and $\pi$ provide a one-to-one correspondence, which leads to the isomorphism of (graded) augmentation varieties $Aug(\Lambda, \F)\simeq Aug(\Sigma_{S^m}\Lambda,\F)$ for $m\geq 2$.

Then we consider the case when $m=1$. From the existence of $i$ and its left inverse $\pi$ and the fact that $\mathcal Q_i(\Lambda)=\emptyset$ for $i<0$ it follows that for every graded augmentation $\varepsilon$ of $\Lambda$ and every generator $c$ of 
$\mathcal A(\Lambda)$ with grading $0$, $\pi^{\ast}(\varepsilon)(c_S)=\varepsilon(c)$. This, together with the fact  that $\pi_1(\Sigma_{S^1}\Lambda)\simeq \pi_1(\Lambda)\times \Z$ (and hence $\pi_1(\Sigma_{S^1}\Lambda)$ has an extra generator compared to $\pi_1(\Lambda)$), implies that there is an identification  
$Aug(\Sigma_{S^1}\Lambda; \mathbb F)\simeq Aug(\Lambda;\mathbb F)\times \mathbb F^{\ast}.$
\end{proof}

Now we assume that $\F=\F_2$, and $\varepsilon$, $\tilde{\varepsilon}$ are augmentations to $\F_2$. Note that in this case we can naturally extend $\tilde{\varepsilon}$ to be an augmentation to $\bar{\F_2}$ by applying the inclusion $\F_2\subset \bar{\F_2}$. 
Here $\bar{\F_2}$ denotes the algebraic closure of $\F_2$.
Now we recall that according to the computation in \cite[Proposition 4.1]{NFAugoTK},
$$Aug(\Lambda_n; \bar{\F_2})\simeq V = \{(a,b)\in \bar{\F_2}^2\ : \ ab\neq -1\},$$
and hence by Proposition \ref{augeqaug0} $Aug(\Sigma_{S^{1}}\Lambda_n,  \bar{\F_2})\simeq V\times \bar{\F_2}^{\ast}$.

From Proposition \ref{nmamforb} combined with Formula \ref{augforspun} and Proposition \ref{augeqaug0} we know that there is an injective algebraic map $$f^{\ast}_{\tilde{L}}:Aug(\tilde{L},\bar{\F_2})\simeq(\bar{\F_2}^{\ast})^3 \to Aug(\Sigma_{S^{1}}\Lambda_n; \bar{\F_2})\simeq V\times \bar{\F_2}^{\ast}.$$

Note that graded augmentation $\tilde{\varepsilon}$ determines a point 
\begin{align*}
(\tilde{\varepsilon}(a^S), \tilde{\varepsilon}(b^S), 
\dots)=(0,0,\dots)\ \mbox{in}\ Aug(\Sigma_{S^{1}}\Lambda_n; \bar{\F_2})\simeq V\times \bar{\F_2}^{\ast}.
\end{align*}
Then we prove the ``stabilized'' version of \cite[Proposition 4.3]{NFAugoTK} 
\begin{proposition}
\label{stabilizedobserv}
Let $k$ be an algebraically closed field, and $V = \{(a,b)\in k^2\ : \ ab\neq -1\}$. There
is no injective algebraic map
$\varphi :(k^{\ast})^2\times (k^{\ast})^m \to V\times (k^{\ast})^m$
having $(0,0,x)$ in its image, where $x\in (k^{\ast})^m$.
\end{proposition}

\begin{proof}
Here we follow the proof of  \cite[Proposition 4.3]{NFAugoTK}. 
Let $(s_1,s_2, t_1\dots, t_m)$ be the coordinates on $(k^{\ast})^{2}\times (k^{\ast})^{m}$, and let $(a, b, c, x_1,\dots,x_m)$ be the coordinates on 
$k^3\times (k^{\ast})^m$, where $V\times (k^{\ast})^m$ is a closed subvariety cut of by the equation $(ab+1)c =1$. Let $A=A(s_1,s_2, t_1\dots, t_m)$, $B=B(s_1,s_2, t_1\dots, t_m)$, $C=C(s_1,s_2, t_1\dots, t_m)$, $X_i=X_i(s_1,s_2, t_1\dots, t_m)$, $i=1,\dots, m$,  be the functions defined by the injective map $\varphi$. Then $A, B, C$, $X_1,\dots,X_m$ satisfy $$(AB+1)C =1.$$ 
Following the same argument as in the proof of  \cite[Proposition 4.3]{NFAugoTK} we can assume that
$$AB = \alpha s^{l_1}_1 s^{l_2}_2 t^{n_1}_1\dots t^{n_m}_m  - 1,$$
where $\alpha\in k^{\ast}$, $A$ and $B$ are polynomials in $k[s_1,s_2, t_1\dots, t_m]$,  $l_i,n_j\geq 0$, where $i=1,2$, $j=1,\dots,m$.
Now we use the fact that there exists certain $p=(s'_1, s'_2, t'_1,\dots, t'_m)$ such that $A(p)B(p) = 0$. 
Then, as in the proof of \cite[Proposition 4.3]{NFAugoTK}, there are two cases. 

The first case is when $l_1,l_2,n_1,\dots,n_m=0$. In this case we see that $\alpha=1$, since otherwise $AB$ would not have any zero.
Therefore, $AB = 0$, which implies that one of $A$ or $B$ is $0$. If $A=0$, then $$\varphi(s_1,s_2, t_1\dots, t_m)=(0,B(s_1,s_2, t_1\dots, t_m),\dots),$$ which contradicts injectivity of $\varphi$. We get to the same contradiction when $B=0$.

The second case concerns the situation when there is at least one non-zero number in $\{l_1,l_2,n_1,\dots,n_m\}$.  Now assume that $char(k)=0$ and, for example, $l_1\neq 0$. 
Then we observe that $$(AB)(s_1,s_2',t'_1,\dots,t'_m)=(\alpha (s'_2)^{l_2} (t'_1)^{n_1}\dots (t'_m)^{n_m}) s^{l_1}_1  - 1.$$
This leads to contradiction since $s'_1$ would be a multiple root of $(AB)(s_1,s_2',t'_1,\dots,t'_m)$ and $(\alpha (s'_2)^{l_2} (t'_1)^{n_1}\dots (t'_m)^{n_m}) s^{l_1}_1  - 1\in k[s_1]$ is separable, in other words $(AB)(s_1,s_2',t'_1,\dots,t'_m)$ has no multiple roots, since it is relatively prime to its formal derivative:
\begin{align*}
(-1)((\alpha (s'_2)^{l_2} (t'_1)^{n_1}\dots (t'_m)^{n_m})s^{l_1}_1  - 1)+\frac{s_1}{l_1}((l_1\alpha(s'_2)^{l_2} (t'_1)^{n_1}\dots (t'_m)^{n_m}))s_1^{l_1-1})=1.
\end{align*}
We can use the same argument in the situation when other numbers from $\{l_1,l_2,n_1,\dots,n_m\}$ are greater than $0$.
The case when $char(k)\neq 0$ completely mimics the corresponding part in the proof of \cite[Proposition 4.3]{NFAugoTK}, and is based on the same formal derivative computation we did in the case when $char(k)=0$.
\end{proof}

Proposition \ref{stabilizedobserv} implies that $\tilde{\varepsilon}$ is not induced by an embedded Maslov number 0 spin exact Lagrangian filling $\tilde{L}$ of $\Sigma_{S^{1}}\Lambda_n$. The same way Proposition \ref{augeqaug0}, Proposition \ref{stabilizedobserv} and Remark \ref{stabh1} imply that $\tilde{\varepsilon}$, which is an augmentation of  $\Sigma_{S^{1}} \dots \Sigma_{S^{1}}\Lambda_n$  defined by the inductive application of $\pi^{\ast}$ to $\varepsilon_{\Lambda_n}$, is not induced by an embedded Maslov number 0 spin exact Lagrangian filling $\tilde{L}$ of $\Sigma_{S^{1}} \dots \Sigma_{S^{1}}\Lambda_n$.
\end{proof}

\begin{remark}
Note that instead of the $S^1$-spinning construction, one can apply the $S^k$-spinning construction, $k\geq 2$, to Classes A and B. In order for the  arguments from Sections \ref{HDAOCA} and \ref{NonSeiClBHD} to work, one needs to know the  analogue of  K\"{u}nneth formula for high dimensional spuns, i.e. analogue of  \cite[Theorem 4.1, Remark 4.2]{GolovkoRizellsphsp}. This  analogue of K\"{u}nneth formula is expected to appear in the forthcoming work of Strako\v{s} \cite{Strakos}.
\end{remark}

\section*{Acknowledgements}
This  project started when the author attended the Differential Geometry and Applications 2022 conference in Hradec Kr\'{a}lov\'{e}. The author would like to thank the organizers of that conference for the hospitality. In addition, the author is grateful to Russell Avdek, Georgios Dimitroglou Rizell, Honghao Gao and Filip  Strako\v{s} for the very helpful  discussions. The author is very grateful to Dan Rutherford for discussing the computation of the linearized complex with $\Z$-coefficients in Section \ref{NonSeiClBHD}. Besides that, the author would like to thank the anonymous referee for suggesting quite a few useful improvements and for finding the gap in one of the arguments of the previous version of the paper.
The author is supported by the GA\v{C}R EXPRO Grant 19-28628X.

\color{black}

\end{document}